\documentclass[11pt,reqno]{amsart}

\usepackage{amsmath,amsthm,amssymb,amscd,url,enumerate,enumitem,mathtools}
\usepackage{graphicx}
\usepackage[margin=1 in]{geometry}
\usepackage{afterpage}
\usepackage{tikz}
\usepackage{pgfplots}
\usepackage{todonotes}
\usepackage[all]{xy}
\usepackage[colorlinks=true,citecolor={blue},linkcolor = purple]{hyperref}
\usepackage{tabularx}
\usepackage{subcaption} 
\usepackage{float}		
\usepackage{booktabs}
\usepackage{caption}
\usepackage{subcaption}
\usepackage{comment}

\usepackage[square]{natbib}
\setcitestyle{numbers}

\definecolor{lightsalmon}{rgb}{1.0, 0.63, 0.48}

\newcommand{\dnd}{\nmid}

\newcommand{\TITLE}{Monogenic Fields Arising from Trinomials}
\newcommand{\TITLERUNNING}{}

\theoremstyle{plain}
\newtheorem{theorem}{Theorem}

\newtheorem{lemma}[theorem]{Lemma}
\newtheorem{corollary}[theorem]{Corollary}

\theoremstyle{definition}

\theoremstyle{remark}
\newtheorem{remark}[theorem]{Remark}

\numberwithin{theorem}{section}

  {\end{list}}

  {\end{list}}

\newcommand{\tightoverset}[2]{%
  \mathop{#2}\limits^{\vbox to -.5ex{\kern-1.05ex\hbox{$#1$}\vss}}}

\renewcommand{\a}{\alpha}

\def\Ocal{{\mathcal O}}

\newcommand{\FF}{\mathbb{F}}

\newcommand{\NN}{\mathbb{N}}

\newcommand{\QQ}{\mathbb{Q}}

\newcommand{\ZZ}{\mathbb{Z}}

\newcommand{\ind}{\operatorname{ind}}

\title[\TITLERUNNING]{\vspace*{-1.3cm} \TITLE}

\date{\today}

\author[R. Ibarra]{Ryan Ibarra}
\author[H. Lembeck]{Henry Lembeck}
\author[M. Ozaslan]{Mohammad Ozaslan}
\author[H. Smith]{Hanson Smith}
\author[K. Stange]{Katherine E. Stange}
\email{Henry.Lembeck@colorado.edu}
\email{Mohammad.Ozaslan@colorado.edu}
\email{Ryan.Ibarra@colorado.edu}
\email{hanson.smith@uconn.edu}
\email{kstange@math.colorado.edu}

\address{
Department of Mathematics, University of Colorado, Campus Box 395, Boulder, Colorado 80309-0395 USA}
\address{Department of Mathematics, University of Connecticut, 341 Mansfield Road U1009
Storrs, CT 06269-1009
USA}


\begin{document}

\begin{abstract}
We call a polynomial \emph{monogenic} if a root $\theta$ has the property that $\mathbb{Z}[\theta]$ is the full ring of integers of $\mathbb{Q}(\theta)$. Consider the two families of trinomials $x^n + ax + b$ and $x^n + cx^{n-1} + d$. For any $n>2$, we show that these families are monogenic infinitely often and give some positive densities in terms of the coefficients. When $n=5$ or 6 and when a certain factor of the discriminant is square-free, we use the Montes algorithm to establish necessary and sufficient conditions for monogeneity, illuminating more general criteria given by Jakhar, Khanduja, and Sangwan using other methods. Along the way we remark on the equivalence of certain aspects of the Montes algorithm and Dedekind's index criterion.
\end{abstract}

\maketitle

\section{Introduction}


Let $K$ be a number field, and denote its ring of integers by $\mathcal{O}_K$. If $\mathcal{O}_K = \ZZ[\theta]$ for some $\theta \in \mathcal{O}_K$, we say that $\mathcal{O}_K$ admits a \emph{power integral basis} or that $K$ is \emph{monogenic}. The classification of monogenic number fields is often known as Hasse's problem.

We use the term \emph{monogenic} to refer to any polynomial $f(x)\in\ZZ[x]$ for which a root $\theta$ has the property that $\ZZ[\theta]$ is the full ring of integers in $\QQ(\theta)$.  Our work seeks to give sufficient conditions for certain polynomials to be monogenic.  By elementary considerations, any polynomial having a square-free discriminant is automatically monogenic.  Both Kedlaya \cite{KedlayaDisc} and Boyd, Martin, and Thom \cite{TriDiscMono} find families of polynomials with square-free discriminant.  We study families with discriminants that are not square-free.

Our main tool in approaching Hasse's problem is the Montes algorithm (for an overview, see \citep{Hanson}; for in-depth treatments, see \cite{Montes} or \cite{Montes2}). We limit ourselves to irreducible trinomials of the form $x^n + ax + b$ or $x^n + cx^{n-1} + d$, with $n=5$ and $6$. Note that the discriminants of these polynomials are not square-free in general (see Theorem \ref{trinomialdisc}).

When a certain factor of the discriminant is square-free, we are able to provide necessary and sufficient conditions 
for the monogeneity of these families 
  (Theorems \ref{QuinticLinearTerm}, \ref{SexticLinearTerm}, \ref{QuinticnMinusOneTerm}, and \ref{SexticnMinusOneTerm}). 
 Using the Montes algorithm to treat the case $n = 4$ has already been studied in \cite{Hanson}. 
Furthermore, we demonstrate infinite families of polynomials (Theorems \ref{DenseLinear} and \ref{DenseNotLinear}) whose roots yield power integral bases for their associated rings of integers infinitely often, namely $x^n + bx + b$ and $x^n + cx^{n-1} + cd$. 
The reader wishing to see the full statements of our results should proceed to Section \ref{Results}.

The literature regarding monogenic fields is extensive. See \cite{GaalsBook} for an extensive and very recent survey of much of the literature; this work has a very in-depth perspective on index form techniques. A general survey can also be found in \cite{Narkiewicz} as well as \cite{EvertseGyoryBook}. Much of the literature focuses on a given degree or Galois group. 
Classically, monogeneity is known for cyclotomic fields and the maximal real subfields thereof.  Gras \cite{Grasprime} shows that, with the exception of maximal real subfields of cyclotomic fields, abelian extensions of prime degree greater than or equal to 5 are not monogenic. Gras \cite{Gras6} also shows that almost all abelian extensions with degree coprime to 6 are not monogenic. Gassert \cite{g17} gives necessary and sufficient conditions for the monogeneity of extensions of the form $x^n+a$; when $n$ is prime see \cite{Westlund}. In \cite{g14}, Gassert investigates the monogeneity of extensions given by shifted Chebyshev polynomials. Jones and Phillips \cite{JPTri} investigate trinomials of the form $x^n+a(m,n)x+b(m,n)$ with $m$ an indeterminate. They find infinitely many distinct monogenic fields and classify the Galois groups, which are either $S_n$ or $A_n$. Although there is overlap with our family $x^n+ax+b$, the methods we employ are distinct. 

As this work was in final edits for release, the authors were made aware of an overlapping recent parallel research line.  
Jakhar, Khanduja, and Sangwan (\cite{JKS1} and \cite{JKS2}) established necessary and sufficient conditions for any trinomial to be monogenic. Their criteria are more general than ours, but our methods are distinct. 
The conditions in our theorems are also more succinct and this allows us to analyze the density of our families; such an analysis is not present in \cite{JKS1} or \cite{JKS2}. 
Concurrently but independent from our work, Jones and White \cite{JonesWhite} prove infinitude and analyze the density of certain families of monogenic trinomials.  In particular, they provide a more complete density theorem than our Theorem \ref{DenseLinear} for trinomials of the form $x^n+bx+b$, but do not address the family in Theorem \ref{DenseNotLinear}.

The outline of the paper is as follows. In Section \ref{Notation} we establish our setup, quote some previous results we will need, and give a very brief overview of part of the Montes algorithm, our main tool in proving these trinomials yield monogenic fields. We will formally state our results in Section \ref{Results}. With Section \ref{Monogeneity} we use the Montes algorithm to prove the roots of the trinomials we are considering yield power integral bases. Section \ref{Infinity} establishes the infinitude of some of our families. Finally, Section \ref{data} contains some computational data for comparison 
to the densities of Theorems \ref{DenseLinear} and \ref{DenseNotLinear}.

\noindent \textbf{Acknowledgments.} The authors would like to thank the Mathematics Department at the University of Colorado Boulder for hosting and supporting the summer 2018 REU that allowed the authors to conduct this research. We also thank Sebastian Bozlee for the help with the code for Section \ref{data}.

\section{Notation, Definitions, and Lemmas}\label{Notation}

In Table \ref{tab:Notation} we outline some standard notation that will be in use throughout the paper.

\begin{figure}[h!]
\begin{table}[H]\caption{Notation}
\begin{center}
\bgroup
\def\arraystretch{1.2}
\begin{tabular}{r c p{10cm} }
\toprule
$K$ &  &  a finite extension of $\QQ$ \\
$\mathcal{O}_K$ &  & the ring of integers of $K$ \\
$\Delta_K$ & & the absolute discriminant of $K$ \\
$f,g, \phi$ & & a monic polynomial in $x$ \\
$\Delta_f$ & & the discriminant of the polynomial $f$ \\
$\theta$ & & a root of a polynomial \\
$\deg f$ & & the degree of the polynomial $f$\\
$a,b,c,\dots$ & & integer coefficients of a polynomial  \\
$p$ & & a prime number \\ 
$v_p$ & & the $p$-adic valuation, normalized so $v_p(p)=1$ \\ 
$\overline{f}$ & & $f$ as viewed in $(\ZZ/m\ZZ)[x]$, when $m$ is clear \\
\bottomrule
\end{tabular}
\egroup
\end{center}
\label{tab:Notation}
\end{table}
\end{figure}

We will need the following well-known result relating field discriminants and polynomial discriminants. Let $f$ be a monic irreducible polynomial of degree $n>1$ and let $\theta$ be a root.  Then
\begin{equation}\label{IndexThm}
\Delta_f=\Delta_K[\Ocal_K:\ZZ[\theta]]^2.
\end{equation}
We can see that it is essential to know the discriminant. For this it is nice to have the following computation of Greenfield and Drucker.

\begin{theorem}\label{trinomialdisc} \cite[Theorem 4]{DiscTri} Consider the trinomial $f(x)=x^n+ax^k+b$. Write $N$ for $\dfrac{n}{\gcd(n,k)}$ and $K$ for $\dfrac{k}{\gcd(n,k)}$. The discriminant of the trinomial is
\[\Delta_f=(-1)^{\frac{n^2-n}{2}}b^{k-1}\left(n^Nb^{N-K}-(-1)^N(n-k)^{N-K}k^Ka^N\right)^{\gcd(n,k)}.\]
\end{theorem}

We now outline the notation necessary for the Montes algorithm. In its full generality the Montes algorithm is a powerful $p$-adic factorization algorithm, but we do not need the full extent of the algorithm for our application to monogeneity. We require a theorem, originally due to Ore \cite{Ore}, that appears in an early step of the algorithm. 

We extend the standard $p$-adic valuation by defining the $p$-adic valuation of $f(x) = a_n x^n + \cdots + a_1 x + a_0 \in \ZZ[x]$ to be 
	\[ v_p(f(x)) = \min_{0 \leq i \leq n} ( v_p(a_i) ). \]
If $\phi(x), f(x) \in \ZZ[x]$ are such that $\deg \phi \leq \deg f$, then we can write
    	\[f(x)=\sum_{i=0}^k a_i(x)\phi(x)^i,\]
for some $k$, where each $a_i(x) \in \ZZ[x]$ has degree less than $\deg \phi$. We call the above expression the \emph{$\phi$-adic development} of $f(x)$. We associate to the $\phi$-adic development of $f$ a Newton polygon by taking the lower convex hull\footnote{Loosely speaking, visualize the points $(i,v_p(a_i(x)))$ as nails in the $xy$-plane and pull a taut string upward from the negative $y$-axis to the positive $y$-axis. We consider the open polygon formed by the string intersecting the nails.} of the integer lattice points $(i,v_p(a_i(x)))$. We call the sides of the Newton polygon with negative slope the \emph{principal $\phi$-polygon}. The number of integer lattice points $(m,n)$, with $m,n>0$, on or under the principal $\phi$-polygon is called the \emph{$\phi$-index of $f$} and denoted $\ind_\phi(f)$. Associated to each side of the principal $\phi$-polygon is a polynomial called the \emph{residual polynomial}. To avoid technicality, we will not define the residual polynomial in general. For our purposes it suffices to note that residual polynomials attached to sides whose only integer lattice points are the initial vertex and terminal vertex are linear polynomials. Again, the interested reader is encouraged to consult \cite{Hanson} for a brief account of the Montes algorithm or \cite{Montes} and \cite{Montes2} for in-depth descriptions and proofs.

Now we state a theorem of Ore \cite{Ore} which will yield our main tool in proving monogeneity.
\begin{theorem}[Ore's theorem of the index]\label{Thmofindex}
  Choose monic polynomials $\phi_1,\dots, \phi_k \in \ZZ[x]$ whose reductions modulo $p$ are exactly the distinct irreducible factors of $\overline{f(x)}$. Let $\theta$ be a root of $f(x)$. Then, 
\[v_p([\Ocal_K:\ZZ[\theta]])\geq \ind_{\phi_1}(f)+\cdots + \ind_{\phi_k}(f).\]
Further, equality holds if, for every $\phi_i$, each side of the principal $\phi_i$-polygon has a separable residual polynomial.
\end{theorem}
For our applications we will employ a clean equivalence derived from Theorem \ref{Thmofindex}.

\begin{corollary}\label{iffCor}
The prime $p$ does not divide $[\Ocal_K:\ZZ[\theta]]$ if and only if $\ind_{\phi_i}(f)=0$ for all $i$. In this case each principal $\phi_i$-polygon is one-sided.
\end{corollary}

\begin{proof}
If $\ind_{\phi_i}(f)>0$, then Theorem \ref{Thmofindex} shows $v_p([\Ocal_K:\ZZ[\theta]])>0$. Conversely, if each $\ind_{\phi_i}(f)=0$, then the associated residual polynomials will all be linear and hence separable. This is because if $\ind_{\phi_i}(f)=0$, then the only integer lattice points on each side of the principal $\phi_i$-polygon are the initial and terminal vertices. Thus $v_p([\Ocal_K:\ZZ[\theta]])=0$ in this case. 

Notice that, since $f(x)\in \ZZ[x]$, if the principal $\phi_i$-polygon has more than one side, then there is necessarily a vertex with positive integer coordinates. This vertex will contribute to $\ind_{\phi_i}(f)$ and ensure that $p$ divides $[\Ocal_K:\ZZ[\theta]]$.
\end{proof}

Before continuing we would like to compare Corollary \ref{iffCor} to another criterion used for studying monogeneity. Consider the following theorem of Dedekind \cite{Dedekind}.

\begin{theorem}[Dedekind's index criterion]\label{Dedind} Let $f(x)$ be a monic, irreducible polynomial in $\ZZ[x]$ and let $\theta$ be a root of $f$. 
If $p$ is a rational prime, we have
\[f(x)\equiv \prod_{i=1}^r\phi_i(x)^{e_i} \bmod p,\]
where the $\phi_i(x)$ are monic lifts of the irreducible factors of $\overline{f(x)}$ to $\ZZ[x]$. Define
\[d(x):=\dfrac{f(x)-\prod\limits_{i=1}^r \phi_i(x)^{e_i}}{p}.  \]
Then $p$ divides $\left[\Ocal_{\QQ(\theta)}:\ZZ[\theta]\right]$ if and only if $\gcd\left(\overline{\phi_i(x)}^{e_i-1},\overline{d(x)}\right)\neq 1$ for some $i$, where we are taking the greatest common divisor in $\FF_p[x]$.
\end{theorem}

\begin{remark}\label{DedIndexEquiv} \textbf{(Equivalence of Polygons and Dedekind for Monogeneity) }
Corollary \ref{iffCor} is equivalent to Dedekind's index criterion. To see this, we consider the $\phi$-adic development, 
\[f(x)=\phi(x)^k+a_{k-1}\phi(x)^{k-1}+\cdots+a_2\phi(x)^2+a_1\phi(x)+a_0.\] For ease of exposition, let $(*)$ denote the condition that either $v_p(a_0)=1$ or $v_p(a_1)=0$. Figure \ref{ExampleMono} shows examples of the principal $\phi(x)$-polygons corresponding to these two cases. Since $\phi(x)$ divides $f(x)$ in $\FF_p[x]$, we have $v_p(a_0)>0$. We notice that $\ind_{\phi}(f)=0$ if and only if condition $(*)$ holds. 
This is because the principal $\phi$-polygon bounds or contains the point $(1,1)$ if and only if $\ind_{\phi}(f)>0$. From the $\phi$-adic development we can translate condition $(*)$: If $v_p(a_0)=1$, then no root of $\phi(x)$ is a root of $f(x)$ modulo $p^2$. If $v_p(a_1)=0$, then the exponent of $\phi(x)$ in the factorization of $f(x)$ modulo $p$ is one. Combining these observations with some algebra, we obtain Dedekind's index criterion. 

We have seen that the Montes algorithm and Dedekind's index criterion are both equally able to answer the question of whether or not a given polynomial in $\ZZ[x]$ is monogenic. Moreover, Newton polygons give us another way of picturing Dedekind's index criterion. It bears noting again that a full application of the Montes algorithm computes the complete factorization of any prime $p$ in $\QQ(\theta)$ and a basis that is $p$-integral, so it is a much more general tool than Dedekind's index criterion.
\end{remark}

\begin{figure}[h!]
\centering
\begin{subfigure}{.5\textwidth}
  \centering
  \begin{tikzpicture}

      \draw[<->] (0,3.5) -- (0,0) -- (3.5,0);
      \draw [fill] (0,3) circle [radius = .05];

      \node [right] at (.15,3) {$v_p\left(a_0(x)\right)$};
      
      \draw [fill] (1,0) circle [radius = .05];
      \node [above] at (2,.1) {$v_p\left(a_1(x)\right)$};
    
      \foreach \x in  {1,2,3,} {
      		\draw (\x, 2pt) -- +(0,-4pt);
            \node [below] at (\x, 0) {\x};
      }
      \foreach \y in  {1,2,3} {
      		\draw (2pt, \y) -- +(-4pt, 0);
            \node [left] at (0,\y) {\y};
      }
      
      \draw (0,3) -- (1,0);
	\end{tikzpicture}
  \caption{The exponent of $\phi(x)$ in $f(x)$ is one}
\end{subfigure}%
\begin{subfigure}{.5\textwidth}
  \centering
  \begin{tikzpicture}

      \draw[<->] (0,3.5) -- (0,0) -- (3.5,0);
      \draw [fill] (0,1) circle [radius = .05];
      d
      
      \node [right] at (.15,1.35) {$v_p\left(a_0(x)\right)$};
      
      \draw [fill] (3,0) circle [radius = .05];
      \node [above] at (3.5,.1) {$v_p\left(a_1(x)\right)$};
    
      \foreach \x in  {1,2,3,} {
      		\draw (\x, 2pt) -- +(0,-4pt);
            \node [below] at (\x, 0) {\x};
      }
      \foreach \y in  {1,2,3} {
      		\draw (2pt, \y) -- +(-4pt, 0);
            \node [left] at (0,\y) {\y};
      }
      
      \draw (0,1) -- (3,0);
	\end{tikzpicture}
  \caption{$\phi(x)$ is not a root modulo $p^2$}
\end{subfigure}
\caption{Examples of principal $\phi(x)$-polygons that could correspond to monogenic polynomials}
\label{ExampleMono}
\end{figure}
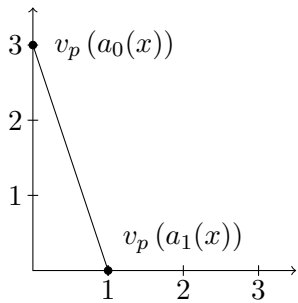
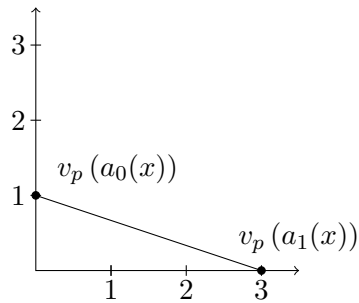

Lastly, in our paper `density' refers to natural density. Let $A \subseteq \NN$ and $a(x) := \# \{a \in A \mid a \leq x \}.$ If 
\[ \lim_{x \to \infty} \frac{a(x)}{x} = \a, \]
we say that $A$ has \emph{natural density} $\alpha$ in $\NN$.

\section{Statements of Results}\label{Results}

Consider the two families $f(x) = x^n + ax + b$ and $g(x) = x^n + cx^{n-1} + d$. The discriminants are 
\[\Delta_f=(-1)^{\frac{n^2-n}{2}}\left(n^nb^{n-1}+(1-n)^{n-1}a^n\right)\] and 
\[\Delta_g=(-1)^{\frac{n^2-n}{2}}d^{n-2}\left(n^nd+(1-n)^{n-1}c^n\right).\] 
We investigate the $n = 5$ and $n = 6$ cases in depth.

\begin{theorem}\label{QuinticLinearTerm}
Let $f(x)=x^5 + ax + b \in \ZZ[x]$ be irreducible and let $\theta$ be a root. Suppose $\frac{2^8a^5 + 5^5b^4}{\gcd(2^8a^5, 5^5b^4)}$ is square-free. Then $\theta$ generates a power integral basis for the ring of integers of $\QQ(\theta)$ if and only if for each prime $p \mid \gcd(2a, 5b)$ one of the following conditions holds:
  \begin{enumerate}[label = (\arabic*)]
  	\item $p \mid a$ and $p \mid b$, but $p^2 \nmid b$.
    \item $p = 2$, $2 \nmid a$, and $a+b\equiv 1 \pmod 4$. 
    \item $p = 5$, $5 \nmid b$, and  $b \not\equiv 1+a, 7+2a, 18+3a, 24+4a \pmod{25}$.
  \end{enumerate}
\end{theorem}

\begin{theorem}\label{SexticLinearTerm}
Let $f(x)=x^6 + ax + b \in \ZZ[x]$ be irreducible and let $\theta$ be a root. Suppose $\frac{ 6^6b^5-5^5a^6}{\gcd(6^6b^5,5^5a^6)}$ is square-free. Then $\theta$ generates a power integral basis for the ring of integers of $\QQ(\theta)$ if and only if for each prime $p \mid \gcd(6b,5a)$ one of the following conditions holds:
\begin{enumerate}
	\item $p \mid a$ and $p \mid b$, but $p^2 \nmid b$.
    \item $p = 2$, $2 \nmid b$, and $a+b\equiv 1 \pmod 4$.
    \item $p = 3$, $3 \nmid b$, and the image of $(a,b)$ in $\left(\ZZ/9\ZZ\right)^2$ is \textbf{not} in the set
\[\left\{(0,1),(0,8),(3,2),(3,5),(6,2),(6,5)\right\}.\]
	\item $p = 5$, $5 \nmid a$, and $a \not\equiv 1 - 4b, 7 + 3b, 18 + 3b, 24 + 4b \pmod{25}$.
\end{enumerate}
\end{theorem}

\begin{theorem}\label{QuinticnMinusOneTerm}
Let $g(x)=x^5 + cx^4 + d \in \ZZ[x]$ be irreducible and $\theta$ a root. Suppose $\frac{5^5d+2^8c^5}{\gcd(5^5d,2^8c^5)}$ is square-free. Then $\theta$ generates a power integral basis for the ring of integers of $\QQ(\theta)$ if and only if $d$ is square-free and if $5\mid c$ but $5\nmid d$, then $c+d \not\equiv 1, 7, 18, 24 \pmod{25}$.
\end{theorem}

\begin{theorem}\label{SexticnMinusOneTerm}
Let $g(x)=x^6 + cx^5 + d \in \ZZ[x]$ be irreducible and $\theta$ a root. Suppose $\frac{6^6d - 5^5c^6}{\gcd(6^6d, 5^5c^6)}$ is square-free. Then $\theta$ generates a power integral basis for the ring of integers of $\QQ(\theta)$ if and only if for every $p\mid \gcd(6d, 5c)$ one of the following conditions hold:
\begin{enumerate}
\item $d$ is square-free.
\item If $2\mid c$ and $2\nmid d$, then $c+d\equiv 1 \pmod 4$. 
\item If $3\mid c$ and $3\nmid d$, then the image of $(c,d)$ in $\left(\ZZ/9\ZZ\right)^2$ is in the set  \[\left\{(3,1),(3,4), (3,7), (6,1), (6,4), (6,7), (0,1), (0,2), (0,4), (0,5).\right\}.\]
\end{enumerate}
\end{theorem}

With sufficient conditions in hand, one can ask about the density of coefficients satisfying these conditions. Naturally, we would like to prove the infinitude of some of the families of monogenic fields. 

\begin{theorem}\label{DenseLinear}
  Fix $n > 2$.
  Let $\theta$ be a root of $f(x) = x^n+bx+b \in \ZZ[x]$.  Then there are infinitely many $b$ such that $f$ is irreducible and $\theta$ generates a power integral basis for the ring of integers of $\QQ(\theta)$. In addition, the density of such $b$ is at least 
	\[\frac{6}{\pi^2}- \left( 1-\frac{6}{\pi^2}\prod_{p\mid (n-1)} \left(1-\frac{1}{p^2} \right)^{-1} \right)>21.58\%. \]
\end{theorem}

\begin{theorem} \label{DenseNotLinear}
	Fix $n > 2$.  Let $c$ be a nonzero integer such that $c \not = \pm 1$ and $c$ is square-free. Suppose $g(x) = x^n + cx^{n-1} + cd \in \ZZ[x]$ is irreducible and let $\theta$ be a root. Consider the quantity
	\[B = \frac{6}{\pi^2} \prod_{p \mid c} \frac{p}{p+1} - \left(1 - \frac{6}{\pi^2}\prod_{p \mid n} \frac{p^2}{p^2 - 1}  \right) . \]
	Then $B$ gives a lower bound on the density of $d$ such that $\theta$ generates a power integral basis for the ring of integers of $\QQ(\theta)$. In particular, if $c$ has exactly one prime factor or has exactly two prime factors and is coprime to $6$, then $B > 0$ and there are infinitely many $d$ yielding such monogenic fields. 
\end{theorem}

\begin{remark} The densities above are merely a byproduct of our proof methods, and appear weak compared to actual densities observed by computation. 
 See Section \ref{data} for these data.
 \end{remark}


\section{Proofs}\label{Monogeneity}
To warm up to Newton polygons, we will prove a well-known result, e.g. \cite[Lemma 2.17]{Narkiewicz}, on the $p$-integrality of polynomials that are Eisenstein at $p$. Proving this here will also simplify the proofs below.
\begin{lemma}\label{pEisenstein}
Suppose $f(x)=x^n+a_{n-1}x^{n-1}+\cdots+a_0$ is Eisenstein at $p$ (each $a_i$ is divisible by $p$ with $p$ dividing $a_0$ exactly once), and let $\theta$ be a root of $f(x)$. Then $p$ does not divide the index $[\Ocal_{\QQ(\theta)}:\ZZ[\theta]]$.
\end{lemma}

\begin{proof}
We have $f(x)\equiv x^n \bmod p$, so we consider only the $x$-adic development of $f(x)$. The $x$-adic development of $f(x)$ is just $f(x)$, so our principal $x$-polygon is one-sided with slope $-\frac{1}{n}$; see Figure \ref{EisenPoly}. No positive integer lattice points lie on or below this polygon, thus the $x$-index is 0. Theorem \ref{Thmofindex} yields the result.

  \begin{figure}[h!]
	\begin{tikzpicture}
        \draw[help lines, color=gray!30, dashed] (0,0) grid (2.5,2.5);
        \draw[help lines, color=gray!30, dashed] (3.5,0) grid (5.5,2.5);

        \draw[<-] (0,2.5) -- (0,0) -- (2.5,0);
        \draw[->] (3.5, 0) -- (5.5,0);
		\draw[dotted] (2.5,0) -- (3.5,0);
        
        \foreach \x in  {1,2,4,5}
     		\draw (\x,2pt) -- +(0,-4pt);
        
		\draw (0,1) -- (2.5, 0.5);
        \draw[dotted] (2.5, 0.5) -- (3.5, 0.3);
        \draw (3.5, 0.3) -- (5, 0);
        
        
        \draw[fill] (0,1) circle [radius = .05];
		\draw[fill] (5,0) circle [radius = .05];
        
        \foreach \x in  {1,2} {
              \node [below] at (\x, 0) {\x};
        }
        \foreach \y in  {1,2} {
              \node [left] at (0,\y) {\y};
        }
        \node [below] at (4,0) {$n-1$};
        \node [below] at (5,-.05) {$n$};
	\end{tikzpicture}
	\caption{The principal $x$-polygon}
	\label{EisenPoly}
  \end{figure}
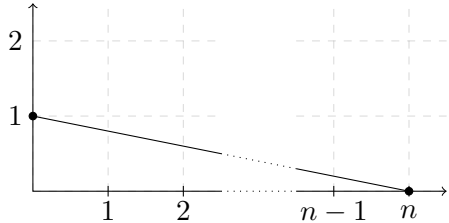
\end{proof}




We will be particularly deliberate with our proof of Theorem \ref{QuinticLinearTerm}. The proofs of the other theorems are very similar, so we will only highlight aspects that are distinct from the proof of Theorem \ref{QuinticLinearTerm}.

\begin{proof}[\bfseries Proof of Theorem \ref{QuinticLinearTerm}] Recall our set-up:  $f(x) = x^5 + ax + b \in \ZZ[x]$ is irreducible, $\theta$ is a root, $K = \QQ(\theta)$, and $\Delta_f=5^5b^4+4^4a^5=3125b^4+256a^5.$ By equation \eqref{IndexThm} and the hypothesis that $\frac{5^5b^4+2^8a^5}{\gcd\left(5^5b^4,2^8a^5\right)}$ is square-free, the only prime factors $p$ of $[\Ocal_K : \ZZ[\theta]]$ are divisors of $\gcd(5b,2a)$. Thus we need only consider the cases given in the theorem statement.




\textbf{\emph{Case 1.}} Suppose $p$ divides $a$ and $b$. Employing the ideas at work in Lemma \ref{pEisenstein}, we see $(1,1)$ is on or under the principal $x$-polygon if and only if $p^2 \mid b$. Corollary \ref{iffCor} shows that in this case $p$ divides the index $[\Ocal_K : \ZZ[\theta]]$ if and only if $p^2\mid b$. 

\textbf{\emph{Case 2.}} Suppose that $p=2$ and $2 \nmid a$. Since $2 \mid 5b$ and $\gcd(2, 5) = 1$, we see $2\mid b$. As a result, 
\[f(x) = x^5 + ax + b \equiv x^5 + ax \equiv x(x^4 + a) \pmod{2}.\]
However, $2 \nmid a$ implies that  $a \equiv 1 \pmod{2}$. Hence 
\[f(x) \equiv x (x^4 + 1) \equiv x (x^4 + 1^4) \equiv x (x + 1)^4 \pmod{2}. \] 
Since the exponent of $x$ is one, it does not contribute to the index. Thus we only need to look at the $(x + 1)$-adic development of $f$, which is
\[ f(x)=(x + 1)^5 - 5(x + 1)^4 + 10 (x + 1)^3 - 10 (x + 1)^2 + (a + 5)(x + 1) + b - a - 1.\] Note $a + 5 \equiv 1 + 5 \equiv 0 \pmod{2}$, so $v_2(a + 5) \geq 1$.
Thus $v_2(b - a - 1)$ is greater than $1$ if and only if the point $(1,1)$ is on or below the principal $(x+1)$-polygon. Thus, applying Corollary \ref{iffCor}, $2\mid [\Ocal_K : \ZZ[\theta]] \Leftrightarrow  v_2(b - a - 1)>1$. 
Hence we wish to ensure that $v_2(b - a - 1) = 1$. Since $a \equiv 1 \pmod{2}$ and $b \equiv 0 \pmod{2}$, we examine four possibilities for the image of $a,b$ in $\ZZ/4\ZZ$. These can seen in Table \ref{table}.
\begin{table}[H]
\begin{center}
  \begin{tabular}{|c | c | c |}
      \hline
      $a$ & $b$ & $b - a - 1$ \\
      \hline
      3 & 2 & 2 \\
      3 & 0 & 0 \\
      1 & 2 & 0 \\
      1 & 0 & 2 \\
      \hline
  \end{tabular}
  \vspace{.1 in}
  \caption{$a$, $b$, and $b-a-1$ modulo 4}
  \label{table}
\end{center}
\end{table}
To visualize Case 2, notice that when $(a,b)$ is congruent to (3,2) or (1,0) modulo 4, i.e., $a+b\equiv 1\pmod 4$, we have the principal $(x+1)$-polygon in Figure \ref{Fig:QuinticLinear3}.
\begin{figure}[h]
  \caption{The principal $(x+1)$-polygon}
  \label{Fig:QuinticLinear3}
  \begin{tikzpicture}
      \draw[help lines, color=gray!30, dashed] (0,0) grid (5.5,2.5);

      \draw[<->] (0,2.5) -- (0,0) -- (5.5,0);
      \draw [fill] (0,1) circle [radius = .05];
      
      \draw [dotted] (1,1) -- (1,2.5);
      \node [right] at (1.05,2) {$v_2(a+5)$};

      \draw (0,1) -- (4,0);
      
      \draw [fill] (2,1) circle [radius = .05];
      \draw [fill] (3,1) circle [radius = .05];
      \draw [fill] (4,0) circle [radius = .05];
      \draw [fill] (5,0) circle [radius = .05];
      
      \foreach \x in  {1,2,3,4,5} {
      		\draw (\x, 2pt) -- +(0,-4pt);
            \node [below] at (\x, 0) {\x};
      }
      \foreach \y in  {1,2} {
      		\draw (2pt, \y) -- +(-4pt, 0);
            \node [left] at (0,\y) {\y};
      }

  \end{tikzpicture}
\end{figure}
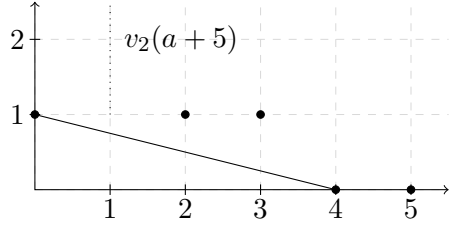
We can see that the integer lattice point corresponding to $(1,v_2(a+5))$ is on the dotted line above the principal $(x+1)$-polygon. 

\textbf{\emph{Case 3.}} Now, suppose that $p = 5$ and $5 \nmid b$. Since $5 \mid 2a$, we see that $5 \mid a$. Thus 
\[f(x) = x^5 + ax + b \equiv x^5 + b \equiv x^5 + b^5 \equiv (x + b)^5 \pmod{5}.\] 
The $(x + b)$-adic development is \[f(x)=(x+b)^5 - 5b(x + b)^4 + 10b^2(x+b)^3 - 10b^3(x+b)^2 + (5b^4 + a)(x+b) - b^5 -ba + b.\] 
We compute $v_5(-b^5 - ba + b)\geq 1$ and $v_5(5b^4 + a)\geq 1$. Thus the lattice point $(1,1)$ will be on or under the principal $(x+b)$-polygon if and only if if $v_5(-b^5 - ba + b)> 1$. Hence, again with Corollary \ref{iffCor}, $5$ divides $[\Ocal_K : \ZZ[\theta]]$ if and only if $25$ divides $-b^5 - ba +b$. 
 
Note that $b_0 = 1, 2, 3, 4$ are solutions to $-x^5 - ax + x \equiv 0 \pmod{5}$. We use Hensel's lemma to obtain solutions modulo $25$, and we find that the solutions $(a,b)$ are of the form $(a, 1 + a), (a, 7 + 2a), (a, 18 + 3a),$ and $(a, 24 + 4a)$. Therefore $25 \mid \left(-b^5 - ba +b\right)$ if and only if $(a,b)$ is one of these pairs. Corollary \ref{iffCor} finishes the argument. 

To see what is happening here, we note that if $(a,b)$ is not of one of the forms above, then we obtain the principal $(x+b)$-polygon in Figure \ref{Fig:QuinticLinear4}.
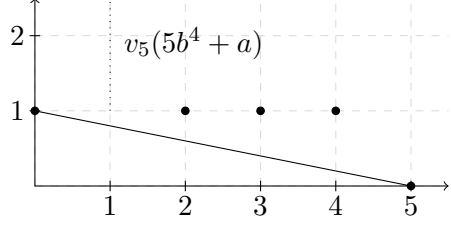
\begin{figure}[h]
  \caption{The principal $(x+b)$-polygon}
  \label{Fig:QuinticLinear4}
  \begin{tikzpicture}
      \draw[help lines, color=gray!30, dashed] (0,0) grid (5.5,2.5);

      \draw[<->] (0,2.5) -- (0,0) -- (5.5,0);
      \draw [fill] (0,1) circle [radius = .05];
      
      \draw [dotted] (1,1) -- (1,2.5);
      \node [right] at (1.05,1.9) {$v_5(5b^4 + a)$};

      \draw (0,1) -- (5,0);
       
      \draw [fill] (2,1) circle [radius = .05];
      \draw [fill] (3,1) circle [radius = .05];
      \draw [fill] (4,1) circle [radius = .05];
      \draw [fill] (5,0) circle [radius = .05];
      
      \foreach \x in  {1,2,3,4,5} {
      		\draw (\x, 2pt) -- +(0,-4pt);
            \node [below] at (\x, 0) {\x};
      }
      \foreach \y in  {1,2} {
      		\draw (2pt, \y) -- +(-4pt, 0);
            \node [left] at (0,\y) {\y};
      }
      
  \end{tikzpicture}
\end{figure}
The integer lattice point $(1,v_5(5b^4 + a))$ lies above the principal $(x+b)$-polygon on the dotted line. 

In conclusion, for all primes $p$ that could possibly divide $[\Ocal_K : \ZZ[\theta] ]$, we have established necessary and sufficient conditions for $v_p([\Ocal_K : \ZZ[\theta]]) = 0$. 
\end{proof}


\begin{proof}[\bfseries Proof of Theorem \ref{SexticLinearTerm}] Recall our set-up:  $f(x) = x^6 + ax + b \in \ZZ[x]$ is irreducible, $\theta$ is a root, $K = \QQ(\theta)$, and $\Delta_f=6^6b^5-5^5a^6=46656b^5-3125a^6$. We assume that $\frac{6^6b^5-5^5a^6}{\gcd\left(6^6b^5,5^5a^6\right)}$ is square-free and consider primes $p$ dividing $\gcd(6b,5a)$. Our approach and Case 1 are exactly analogous to the proof of Theorem \ref{QuinticLinearTerm}.

\textbf{\emph{Case 2.}} 
    Suppose $p = 2$ and $2 \nmid b$. We see $2 \mid a$ and as a result
		\[ f(x) = x^6 + ax + b = x^6 + b \equiv (x^3 + b)^2 \pmod{2}. \]
Furthermore $b \equiv 1 \pmod 2$, so 
  \[ f(x) \equiv (x^3 + 1)^2 \equiv \left[(x + 1)(x^2 + x + 1)\right]^2 \pmod{2}. \]
We have two irreducible factors to consider. For the irreducible factor $(x + 1)$, the $(x + 1)$-adic development of $f(x)$ is 
    \begin{align*}
    (x + 1)^6 	&- 6(x + 1)^5 + 15(x + 1)^4 - 20(x + 1)^3 \\ 
    			&+ 15(x + 1)^2 + (a - 6)(x + 1) - a + b + 1.
    \end{align*}
We observe that $a_1(x) = a - 6 \equiv 0 \pmod{2}$, so $v_2(a_1(x)) \geq 1$. We want to ensure $4 \nmid a_0(x) = -a + b + 1$. Thus $(a,b)$ must be equivalent to $(2,3)$ or $(0,1)$ modulo 4.
 
We turn our attention to the other irreducible factor, $x^2 + x + 1$. The $(x^2 + x + 1)$-adic development of $f$ is 
 	\[(x^2 + x + 1)^3 - 3x(x^2 + x + 1)^2 + (2x - 2)(x^2 + x + 1) + ax + b + 1.\]
It is clear that $a_1(x) = 2x - 2 \equiv 0 \pmod{2}$, so $v_2(a_1(x)) \geq 1$. We need to ensure that $4 \nmid a_0(x) = ax + b + 1$. Thus we need either $v_2(a)=1$ or $b\equiv 1 \pmod 4$. 

Since the conditions coming from the irreducible factor $(x+1)$ are more restrictive, we conclude that $2\dnd[\Ocal_K : \ZZ[\theta] ]$ if and only if $(a,b)$ is equivalent to either $(2,3)$ or $(0,1)$ modulo 4.

\textbf{\emph{Case 3.}} Suppose $p = 3$ and $3 \nmid b$. Since $3 \mid a$, we have
	\[ f(x) = x^6 + ax + b \equiv x^6 + b \equiv (x^2 + b)^3 \pmod{3}. \]
There are two subcases.

\textbf{\emph{Subcase 3.1.}} Suppose $b \equiv 1$ modulo 3. Then, $x^2 + b \equiv x^2 + 1$, which is irreducible. Hence, $f(x) \equiv (x^2 + 1)^3 \pmod{3}$. The $(x^2 + 1)$-adic development of $f$ is 
    \[(x^2 + 1)^3 - 3(x^2 + 1)^2 + 3(x^2 + 1) + ax + b - 1 .\]
Clearly, $v_3(a_1(x)) = 1$. In this subcase, $v_3([\Ocal_K : \ZZ[\theta] ]) = 0 \Leftrightarrow v_3(ax + b - 1) =1$. Translating, either $a \not \equiv 0 \pmod{9}$ or $b \not \equiv 1 \pmod{9}$.

\textbf{\emph{Subcase 3.2.}} Suppose $b \equiv 2$ modulo 3. Then, $x^2 + b \equiv x^2 + 2 \equiv (x + 1)(x + 2) \pmod{3}$, and 
	\[ f(x) \equiv ((x+1)(x+2))^3 \pmod{3}. \]
First, we examine the factor $(x+1)$. The $(x+1)$-adic development of $f$ is 
\begin{align*} 
(x+1)^6 	&- 6(x+1)^5 + 15(x+1)^4 - 20(x+1)^3 \\ 
 			& + 15(x+1)^2 + (a - 6)(x+1) - a + b + 1.
\end{align*}
We observe that $v_3(a - 6) \geq 1$. To avoid $-a + b + 1 \equiv 0 \pmod{9}$ is to require that $a \not\equiv b + 1 \pmod{9}$.

    Lastly, we look at the factor $(x+2)$. The $(x+2)$-adic development of $f$ is 
   \begin{align*} 
   	(x+2)^6 &- 12(x+2)^5 + 60(x+2)^4 - 160(x+2)^3 \\
    		&+ 240(x+2)^2 + (a - 192)(x+2) - 2a + b + 64. 
    \end{align*}
Since $v_3(a-192) \geq 1$, we want $-2a + b + 64 \not\equiv 0 \pmod{9}$. This happens exactly when $b \not\equiv 2a - 64 \equiv 2a - 1 \pmod{9}$.
 
Therefore, in this subcase, we see $v_3([\Ocal_K : \ZZ[\theta] ]) = 0$ if and only if the reduction of $(a,b)$ modulo 9 is not a member of the set
\[\left\{(0,1),(0,8),(3,2),(3,5),(6,2),(6,5)\right\}.\]

\textbf{\emph{Case 4.}} Finally, assume that $p = 5$ and $5 \nmid a$. Since $5\mid b$, we have 
\[f(x)\equiv x^6 + ax\equiv x(x+a)^5.\] The only irreducible factor that concerns us is $x+a$. The $(x+a)$-adic development of $f(x)$ is 
    	\begin{align*}
          (x+a)^6 	& - 6a(x+a)^5 + 15a^2(x+a)^4 - 20a^3(x+a)^3 \\ 
          			&+ 15a^4(x+a)^2 + (-6a^5 + a)(x+a) + a^6 - a^2 + b.
        \end{align*}
Proceeding in the same manner as in the proof of Theorem \ref{QuinticLinearTerm}, $v_5([\Ocal_K : \ZZ[\theta] ]) = 0$ if and only if $a$ is not of the form $1 - 4b, 7 + 3b, 18 + 3b, \text{ or } 24 + 4b \pmod {25}.$

For all primes $p$ which could possible divide $[\Ocal_K : \ZZ[\theta]]$, we have shown the necessity and sufficiency of our conditions.
\end{proof}

\begin{proof}[\bfseries Proof of Theorem \ref{QuinticnMinusOneTerm}]
Recall that we are considering the irreducible polynomial $g(x)=x^5 + cx^4 + d \in \ZZ[x]$. 
One computes $\Delta_g=d^3\left(5^5d+4^4c^5\right)=d^3\left(3125d+256c^5\right)$. We assume that $\dfrac{5^5d+2^8c^5}{\gcd\left(5^5d,2^8c^5\right)}$ is square-free. The possible prime divisors of the index are primes dividing $d$ or $\gcd(5d,2c)$. The only prime we may have to consider that is not necessarily a divisor of $d$ is 5.

For primes $p\mid d$ we have 
\[g(x)\equiv x^5+cx^4\equiv x^4(x+c)\pmod p.\]
As the exponent is one, the factor $x+c$ contributes nothing to the index. 
The $x$-adic development is again $g(x)$. By the standard argument, $v_p([\Ocal_K : \ZZ[\theta] ]) = 0$ if and only if $d$ is square-free. 
 
\pmb{$p=5$:}  Suppose now that $5\mid c$ and $5 \nmid d$. We have 
\[f(x) = x^5 + cx^4 + d \equiv x^5 + d \equiv x^5 + d^5 \equiv (x + d)^5 \pmod{5}.\] 
The $(x + d)$-adic development of $f(x)$ is given by 
\begin{align*}
(x+d)^5 &+ (c-5d)(x + d)^4 + (10d^2-4cd)(x+d)^3 \\
&+ (6cd^2-10d^3)(x+d)^2 + (5d^4 - 4cd^3)(x+d) + cd^4 + d - d^5.
\end{align*} 
As we expect, 
$5\mid [\Ocal_K : \ZZ[\theta]] \Leftrightarrow 25 \mid \left(cd^4 + d - d^5\right)$. One computes that $25$ divides $cd^4 + d - d^5$ if and only if $c+d\equiv 1,7,18,$ or $24 \pmod {25}$. 
\end{proof}

\begin{proof}[\textbf{Proof of Theorem \ref{SexticnMinusOneTerm}}] 

We remind ourselves that we are considering $g(x)=x^6+cx^5+d$. We have $\Delta_g=-d^4(6^6d-5^5c^6)$ and by hypothesis $\frac{6^6d-5^5c^6}{\gcd(6^6d,5^5c^6)}$ is square-free. We consider primes $p$ dividing $d$ or $\gcd(6d,5c)$. The only primes we may have to consider that are not divisors of $d$ are 2 and 3.

Our routine argument shows that prime divisors of $d$ divide the index if and only if their square divides $d$.

\pmb{$p=2$:}  Suppose $2\mid c$ and $2\nmid d$. Reducing yields \[g(x)=x^6+cx^5+d\equiv (x+1)^2(x^2+x+1)^2 \pmod 2.\] The $(x+1)$-adic development is 
\begin{align*}
g(x)=(x+1)^6+&(-6+c)(x+1)^5+(15-5c)(x+1)^4+(-20+10c)(x+1)^3\\
&+(15-10c)(x+1)^2+(-6+5c)(x+1)-c+d+1.
\end{align*}
This factor not contributing to the index is equivalent to $(c,d)$ reducing to either $(0,1)$ or $(2,3)$ in $\left(\ZZ/4\ZZ\right)^2$. 

Continuing, the $(x^2+x+1)$-adic development is 
\begin{align*}
g(x)=(x^2+x+1)^3+&(-3x+cx-2c)(x^2+x+1)^2\\
&+(2x+cx+3c-2)(x^2+x+1)-cx-c+d+1.
\end{align*}
We are concerned with $v_2(-cx-c+d+1)=\min(v_2(-c),v_2(-c+d+1))$. For this factor not to contribute to the index it is necessary and sufficient that $c\equiv 2 \pmod 4$ or $d\equiv 1 \pmod 4$. The $(x+1)$-adic development is more restrictive, so $v_2([\Ocal_K : \ZZ[\theta] ]) = 0$ if and only if $(c,d)$ reduces to either $(0,1)$ or $(2,3)$ in $\left(\ZZ/4\ZZ\right)^2$.

\pmb{$p=3$:}  Suppose $3\mid c$ and $3\nmid d$. If $d\equiv 1 \pmod 3$, then $g(x)\equiv (x^2+1)^3 \pmod 3$. The $(x^2+1)$-adic development is 

\begin{align*}
g(x)=(x^2+1)^3&+(cx-2c-3)(x^2+x+1)^2\\
&+(6x+cx+3c)(x^2+x+1)-cx-c+d+2.
\end{align*}
Thus $v_3([\Ocal_K : \ZZ[\theta] ]) = 0$ if and only if either $c\equiv 3,6 \pmod 9$ or $d\equiv 1,4 \pmod 9$.

If $d\equiv 2$ modulo 3, then $g(x)\equiv (x+1)^3(x-1)^3$ modulo 3. The $(x+1)$-adic development is above. The $(x-1)$-adic development is 
\begin{align*}
g(x)=(x-1)^6+&(6+c)(x-1)^5+(15+5c)(x-1)^4+(20+10c)(x-1)^3\\
&(15+10c)(x-1)^2+(6+5c)(x-1)+1+c+d.
\end{align*}
Combining this with the conditions coming from the $(x+1)$-adic development it is necessary and sufficient that $v_3(d-c+1)=v_3(d+c+1)=1$. Therefore $v_3([\Ocal_K : \ZZ[\theta] ]) = 0$ if and only if the image of $(c,d)$ in $\left(\ZZ/9\ZZ\right)^2$ is either $(0,2)$ or $(0,5)$.
\end{proof}

\section{Infinitude of the Families}\label{Infinity}

In this section we will let $n\geq 2$, but restrict the coefficients of our families and find that they are monogenic infinitely often. To do this, we will actually prove that the coefficients yielding monogenic fields have positive density in $\ZZ$. This requires considering the density of square-free values of parts of the discriminant. 

In general, showing a polynomial takes on infinitely many square-free values can be difficult:  for example, it is not known whether there is a single quartic polynomial that is square-free infinitely often \cite{BookerBrowning}.  In our case, we only require some results on linear polynomials. The first is a result from Prachar \cite{Prachar} about the density of square-free integers congruent to $m$ modulo $k$. Let $S(x;m,k)$ denote the number of square-free integers not exceeding $x$ that are congruent to $m$ modulo $k$.

\begin{theorem}\label{Prachar}
If $\gcd(m,k)=1$ and $k\leq x^{\frac{2}{3}-\epsilon}$, then 
$$S(x;m,k)\sim \frac{6x}{\pi^2k}\prod_{p\mid k} \left(1-\frac{1}{p^2} \right)^{-1}\quad\quad (x\rightarrow\infty).$$
\end{theorem}

We will also need to know the number of integers not exceeding $x$ that are square-free and coprime to $k$. Denote this quantity by $T(x;k)$. The following is a straightforward corollary of Theorem \ref{Prachar} if one notes there are $\phi(k)=k\prod_{p\mid k}\frac{p-1}{p}$ distinct congruence classes modulo $k$ that are relatively prime to $k$.

\begin{corollary}\label{CoprimeDensity}
With the notation as above

\[T(x;k) \sim \frac{6x}{\pi^2} \prod_{p \mid k} \frac{p}{p+1} \quad\quad (x \to \infty).\]

\end{corollary}

\begin{proof}[\bfseries Proof of Theorem \ref{DenseLinear}]
We are considering the polynomial $f(x) = x^n+bx+b$. First note that if $b$ is square-free, then $f$ is irreducible by Eisenstein's Criterion. Recall that the density of square-free $b$ is $\frac{6}{\pi^2}$. We have computed $\Delta_f=\pm b^{n-1}(n^n+(1-n)^{n-1}b)$, and Lemma \ref{pEisenstein} tells us that if $n^n+(1-n)^{n-1}b$ is square-free in addition to $b$, the number field generated by $f(x)$ is monogenic. Notice that choices of $b\in \ZZ$ are in bijection with integers congruent to $n^n$ modulo $(n-1)^{n-1}$. We take $S(x;n^n,(n-1)^{n-1})$. By Theorem \ref{Prachar}, the density of $b$ such that $n^n + (1-n)^{n-1}b$ is square-free is
\[
\frac{6}{\pi^2}\prod_{p\mid (n-1)} \left( 1-\frac{1}{p^2} \right)^{-1}>\frac{6}{\pi^2}.
\] 
Note that we do not have a factor of $(n-1)^{n-1}$ in the denominator of $\frac{6}{\pi^2}$ because we are considering the density of square-free integers congruent to $n^n$ among all integers congruent to $n^n$. 
For the purposes of a lower bound on the density of 
\[
  \left\{ b \in \ZZ : \mbox{$b$ and $n^n+(1-n)^{n-1}b$ are square-free} \right\},
\] 
the worst case scenario is that every value of $n^n + (1-n)^{n-1}b$ that is not square-free occurs when $b$ is square-free.  Thus the density of square-free $b$ with $n^n+(1-n)^{n-1}b$ also square-free is at least 
\[\frac{6}{\pi^2}- \left(1-\frac{6}{\pi^2}\prod_{p\mid (n-1)} \left( 1 - \frac{1}{p^2} \right)^{-1} \right) > \frac{6}{\pi^2} - \left( 1 - \frac{6}{\pi^2} \right) \approx 21.58\%.\]
\end{proof} 

\begin{proof}[\bfseries Proof of Theorem \ref{DenseNotLinear}]
    Consider $g(x) = x^n + cx^{n-1} + cd$ with $n$ and $c$ fixed such that $c$ is square-free, $c \neq \pm 1$, and $\gcd(c,n)=1$. Since $c\neq \pm 1$, Eisenstein's criterion shows $g$ is irreducible when $\gcd(c,d)=1$. 
    
    We wish to analyze the density of $d$ for which $g$ is monogenic. 
	First, for $g$ to be monogenic it is necessary that $\gcd(c,d)=1$. We have computed $\Delta_g=\pm(cd)^{n-2}\left(n^n(cd)+(1-n)^{n-1}c^n\right)$. If $\gcd(c,d)=1$, the product $cd$ is square-free if $d$ is square-free. If $d$ and $cd n^n+(1-n)^{n-1} c^n$ are square-free, then Lemma \ref{pEisenstein} shows that $g$ is monogenic. The factor of $c$ is extraneous, so we will investigate when $d n^n+(1-n)^{n-1} c^{n-1}$ is square-free. 

From Corollary \ref{CoprimeDensity}, the proportion of square-free $d$ that are coprime to $c$ is given by
    \[\frac{6}{\pi^2} \prod_{p \mid c} \frac{p}{p+1}. \]
In addition, Theorem \ref{Prachar} tells us the density of square-free integers among all integers congruent to  $(1-n)^{n - 1} c^{n-1}$ modulo $n^n$ is
    \[\frac{6}{\pi^2} \prod_{p \mid n} \frac{p^2}{p^2 - 1}. \]
Thus (much as in the last proof) a lower bound on the density of square-free $d$ coprime to $c$ such that $(1-n)^{n - 1} c^{n-1} + d n^n$ is square-free is,
\begin{equation*}\label{toughdensity}
B=\frac{6}{\pi^2} \prod_{p \mid c}\frac{p}{p+1}  - \left(1 - \frac{6}{\pi^2}\prod_{p \mid n} \frac{p^2}{p^2 - 1} \right)=\frac{6}{\pi^2}\left(\prod_{p \mid c}\frac{p}{p+1}+\prod_{p \mid n} \frac{p^2}{p^2 - 1}\right)-1.
\end{equation*}
This is non-trivial if this value is positive.

If $c$ has exactly one prime factor, then 
\[
  \frac{6}{\pi^2} \prod_{p \mid c}\frac{p}{p+1}\geq \frac{6\cdot 2}{\pi^2\cdot 3}.
\]
Hence 
\[B> \frac{6\cdot 2}{\pi^2\cdot 3}  - \left(1 - \frac{6}{\pi^2}\right)\approx 0.013.\]
On the other hand, if $c$ has at most two prime factors and is coprime to $6$, then we have
\[
  \frac{6}{\pi^2} \prod_{p \mid c}\frac{p}{p+1}\geq \frac{6\cdot 5 \cdot 7}{\pi^2\cdot 6 \cdot 8}.
\]
Hence 
\[B> \frac{6\cdot 5\cdot 7}{\pi^2\cdot 6\cdot 8}  - \left(1 - \frac{6}{\pi^2}\right)\approx 0.051.\]
\end{proof}

In both proofs, the densities would be improved if we could assume that the square-freeness of the two relevant quantities is independent. If this was the case, then we would expect a lower bound for the family of Theorem \ref{DenseLinear} of at least
\[
  B_1 := \frac{36^2}{\pi^4}\prod_{p \mid (n-1)} \left( 1 - \frac{1}{p^2} \right)^{-1}
\]
and for the family of Theorem \ref{DenseNotLinear} of at least
\[
  B_2 := \frac{6^2}{\pi^4} \prod_{p \mid c} \frac{p}{p+1} \prod_{p \mid n} \frac{p^2}{p^2-1}.
\]
Note that, using $\prod_p \left( 1-\frac{1}{p^2} \right)^{-1} = \zeta(2)$,
\[
  0.28 \approx \frac{27}{\pi^4} \le B_1 \le  \frac{6}{\pi^2} \approx 0.61.
\]
Similarly, we have
\[
  0 \le B_2 \le \frac{6^2}{\pi^4} \prod_p \left( 1 + \frac{1}{p^2-1} \right) \le \frac{72}{\pi^4}  \approx 0.74.
\]

\section{Computational Data}\label{data}

Table \ref{fig:data} is for comparison to Theorems \ref{DenseLinear} and \ref{DenseNotLinear}.  We can see it is rare that one of our trinomials yields a monogenic field for which it is not a generator. It is also noteworthy that our theorems on the monogeneity of the trinomials $f$ and $g$ do not capture all instances in which $f$ and $g$ yield monogenic fields. Specifically, there are instances when the relevant factors of $\Delta_f$ and $\Delta_g$ are not square-free, but those square factors do not contribute to the index. It does not appear the machinery we use is adequate to understand when and why these square factors do not contribute to the index. 

\begin{figure}[h!]
\begin{center}
\begin{table}[H] \caption{Monogenic Percentages for Degrees 5 and 6}
\begin{center}
\bgroup
\def\arraystretch{1.1}
\begin{tabular}{c  c  c  c  c} 
& \% & & \% with $\theta$ & \% satisfying hypotheses\\
Family & monogenic & & a generator & of relevant Theorem\\
\toprule
$x^5+bx+b$ & 52.46 & & 50.50 & 50.50\\
$x^6+bx+b$ & 58.49 & & 57.71 & 57.71\\

$x^5+cx^4+c$ & 44.84 & & 43.10 & 35.92\\
$x^6+cx^5+c$ & 58.68 & & 58.00 & 29.00\\

\hline

$x^5+ax+b$ & 61.17 & & 60.86 & 60.86 \\
$x^6+ax+b$ & 61.10 & & 60.90 & 60.90 \\

$x^5+cx^4+d$ & 55.78 & & 51.80 & 51.80\\
$x^6+cx^5+d$ & 45.43 & & 44.66 & 26.00 \\

\hline

$x^5+2x^4+2d$ & 36.88 & & 33.67 & 33.67 \\
$x^5+3x^4+3d$ & 43.96 & & 43.29 & 43.29 \\
$x^5+4x^4+4d$ & 65.97 & & 0.00 & 0.00 \\
$x^5+5x^4+5d$ & 42.19 & & 42.08 & 42.08\\
$x^5+6x^4+6d$ & 32.05 & & 28.85 & 28.85\\
$x^5+7x^4+7d$ & 45.18 & & 45.13 & 45.13\\
$x^5+8x^4+8d$ & 13.08 & & 0.00 & 0.00\\

\hline

$x^6+2x^5+2d$ & 40.29 & & 38.48 & 38.48 \\
$x^6+3x^5+3d$ & 43.53 & & 43.28 & 14.43\\
$x^6+4x^5+4d$ & 2.50 & & 0.00 & 0.00\\
$x^6+5x^5+5d$ & 48.11 & & 48.09 & 16.03\\
$x^6+6x^5+6d$ & 30.38 & & 28.86 & 28.84\\
$x^6+7x^5+7d$ & 51.57 & & 51.57 & 17.19\\
$x^6+8x^5+8d$ & 4.91 & & 0.00 & 0.00\\



\bottomrule
\end{tabular}
\caption*{\small{For families with a single parameter $a,b,c,$ or $d$, the values tested were $[-500000,500000]$. For families with two parameters the values tested were $[-500,500]$. The percentages are rounded to the nearest hundredth.}}
\egroup
\end{center}
\label{fig:data}
\end{table}
\end{center}
\end{figure}

\bibliography{bibliography}{}

\begin{thebibliography}{22}
\providecommand{\natexlab}[1]{#1}
\providecommand{\url}[1]{\texttt{#1}}
\expandafter\ifx\csname urlstyle\endcsname\relax
  \providecommand{\doi}[1]{doi: #1}\else
  \providecommand{\doi}{doi: \begingroup \urlstyle{rm}\Url}\fi

\bibitem[Booker and Browning(2016)]{BookerBrowning}
A.~R. Booker and T.~D. Browning.
\newblock Square-free values of reducible polynomials.
\newblock \emph{Discrete Anal.}, Paper No. 8, 16, 2016.
\newblock ISSN 2397-3129.
\newblock \doi{10.19086/da.732}.
\newblock URL \url{https://doi.org/10.19086/da.732}.

\bibitem[Boyd et~al.(2015)Boyd, Martin, and Thom]{TriDiscMono}
D.~W. Boyd, G.~Martin, and M.~Thom.
\newblock Squarefree values of trinomial discriminants.
\newblock \emph{LMS J. Comput. Math.}, 18\penalty0 (1):\penalty0 148--169,
  2015.
\newblock ISSN 1461-1570.
\newblock \doi{10.1112/S1461157014000436}.
\newblock URL \url{https://doi.org/10.1112/S1461157014000436}.

\bibitem[Dedekind(1878)]{Dedekind}
R.~Dedekind.
\newblock \"{U}ber den {Z}usammenhang zwischen der {T}heorie der {I}deale und
  der {T}heorie der h\"{o}heren {K}ongruenzen.
\newblock \emph{G\"{o}tt. Abhandlungen}, 23:\penalty0 3--38, 1878.

\bibitem[Evertse and Gy{\H{o}}ry(2017)]{EvertseGyoryBook}
J.-H. Evertse and K.~Gy{\H{o}}ry.
\newblock \emph{Discriminant equations in {D}iophantine number theory},
  volume~32 of \emph{New Mathematical Monographs}.
\newblock Cambridge University Press, Cambridge, 2017.
\newblock ISBN 978-1-107-09761-2.
\newblock \doi{10.1017/CBO9781316160763}.
\newblock URL \url{https://doi.org/10.1017/CBO9781316160763}.

\bibitem[Fadil et~al.(2012)Fadil, Montes, and Nart]{Montes}
L.~E. Fadil, J.~Montes, and E.~Nart.
\newblock Newton polygons and p-integral bases of quartic number fields.
\newblock \emph{Journal of Algebra and Its Applications}, 11\penalty0
  (04):\penalty0 1250073, 2012.
\newblock \doi{10.1142/S0219498812500739}.
\newblock URL
  \url{https://www.worldscientific.com/doi/abs/10.1142/S0219498812500739}.

\bibitem[Ga\'{a}l(2019)]{GaalsBook}
I.~Ga\'{a}l.
\newblock \emph{Diophantine equations and power integral bases}.
\newblock Birkh\"{a}user/Springer, Cham, 2019.
\newblock ISBN 978-3-030-23864-3; 978-3-030-23865-0.
\newblock \doi{10.1007/978-3-030-23865-0}.
\newblock URL \url{https://doi.org/10.1007/978-3-030-23865-0}.
\newblock Theory and algorithms, Second edition of [ MR1896601].

\bibitem[Gassert(2014)]{g14}
T.~A. Gassert.
\newblock Discriminants of {C}hebyshev radical extensions.
\newblock \emph{J. Th\'eor. Nombres Bordeaux}, 26\penalty0 (3):\penalty0
  607--634, 2014.
\newblock ISSN 1246-7405.
\newblock URL \url{http://jtnb.cedram.org/item?id=JTNB_2014__26_2_607_0}.

\bibitem[Gassert(2017)]{g17}
T.~A. Gassert.
\newblock A note on the monogeneity of power maps.
\newblock \emph{Albanian J. Math.}, 11\penalty0 (1):\penalty0 3--12, 2017.
\newblock ISSN 1930-1235.

\bibitem[Gras(1986{\natexlab{a}})]{Gras6}
M.-N. Gras.
\newblock Condition n\'ecessaire de monog\'en\'eit\'e de l'anneau des entiers
  d'une extension ab\'elienne de {${\bf Q}$}.
\newblock In \emph{S\'eminaire de th\'eorie des nombres, {P}aris 1984--85},
  volume~63 of \emph{Progr. Math.}, pages 97--107. Birkh\"auser Boston, Boston,
  MA, 1986{\natexlab{a}}.

\bibitem[Gras(1986{\natexlab{b}})]{Grasprime}
M.-N. Gras.
\newblock Non monog\'{e}n\'{e}it\'{e} de l'anneau des entiers des extensions
  cycliques de {${\bf Q}$} de degr\'{e} premier {$l\geq 5$}.
\newblock \emph{J. Number Theory}, 23\penalty0 (3):\penalty0 347--353,
  1986{\natexlab{b}}.
\newblock ISSN 0022-314X.
\newblock \doi{10.1016/0022-314X(86)90079-X}.
\newblock URL \url{https://doi.org/10.1016/0022-314X(86)90079-X}.

\bibitem[Greenfield and Drucker(1984)]{DiscTri}
G.~R. Greenfield and D.~Drucker.
\newblock On the discriminant of a trinomial.
\newblock \emph{Linear Algebra Appl.}, 62:\penalty0 105--112, 1984.
\newblock ISSN 0024-3795.
\newblock \doi{10.1016/0024-3795(84)90089-2}.
\newblock URL \url{https://doi.org/10.1016/0024-3795(84)90089-2}.

\bibitem[Gu\`{a}rdia et~al.(2015)Gu\`{a}rdia, Montes, and Nart]{Montes2}
J.~Gu\`{a}rdia, J.~Montes, and E.~Nart.
\newblock Higher newton polygons and integral bases.
\newblock \emph{Journal of Number Theory}, 147:\penalty0 549 -- 589, 2015.
\newblock ISSN 0022-314X.
\newblock \doi{https://doi.org/10.1016/j.jnt.2014.07.027}.
\newblock URL
  \url{http://www.sciencedirect.com/science/article/pii/S0022314X14002777}.

\bibitem[Jakhar et~al.(2016)Jakhar, Khanduja, and Sangwan]{JKS1}
A.~Jakhar, S.~K. Khanduja, and N.~Sangwan.
\newblock On prime divisors of the index of an algebraic integer.
\newblock \emph{J. Number Theory}, 166:\penalty0 47--61, 2016.
\newblock ISSN 0022-314X.
\newblock \doi{10.1016/j.jnt.2016.02.021}.
\newblock URL \url{https://doi.org/10.1016/j.jnt.2016.02.021}.

\bibitem[Jakhar et~al.(2017)Jakhar, Khanduja, and Sangwan]{JKS2}
A.~Jakhar, S.~K. Khanduja, and N.~Sangwan.
\newblock Characterization of primes dividing the index of a trinomial.
\newblock \emph{Int. J. Number Theory}, 13\penalty0 (10):\penalty0 2505--2514,
  2017.
\newblock ISSN 1793-0421.
\newblock \doi{10.1142/S1793042117501391}.
\newblock URL \url{https://doi.org/10.1142/S1793042117501391}.

\bibitem[Jones and Phillips(2018)]{JPTri}
L.~Jones and T.~Phillips.
\newblock Infinite families of monogenic trinomials and their {G}alois groups.
\newblock \emph{Internat. J. Math.}, 29\penalty0 (5):\penalty0 1850039, 11,
  2018.
\newblock ISSN 0129-167X.
\newblock \doi{10.1142/S0129167X18500398}.
\newblock URL \url{https://doi.org/10.1142/S0129167X18500398}.

\bibitem[{Jones} and {White}(2019)]{JonesWhite}
L.~{Jones} and D.~{White}.
\newblock {Monogenic trinomials with non-squarefree discriminant}.
\newblock \emph{arXiv e-prints}, art. arXiv:1908.07947, Aug 2019.

\bibitem[Kedlaya(2012)]{KedlayaDisc}
K.~S. Kedlaya.
\newblock A construction of polynomials with squarefree discriminants.
\newblock \emph{Proc. Amer. Math. Soc.}, 140\penalty0 (9):\penalty0 3025--3033,
  2012.
\newblock ISSN 0002-9939.
\newblock \doi{10.1090/S0002-9939-2012-11231-6}.
\newblock URL \url{https://doi.org/10.1090/S0002-9939-2012-11231-6}.

\bibitem[Narkiewicz(2004)]{Narkiewicz}
W.~Narkiewicz.
\newblock \emph{Elementary and analytic theory of algebraic numbers}.
\newblock Springer Monographs in Mathematics. Springer-Verlag, Berlin, third
  edition, 2004.
\newblock ISBN 3-540-21902-1.
\newblock \doi{10.1007/978-3-662-07001-7}.
\newblock URL \url{https://doi.org/10.1007/978-3-662-07001-7}.

\bibitem[Ore(1928)]{Ore}
{\O}.~Ore.
\newblock Newtonsche {P}olygone in der {T}heorie der algebraischen
  {K}\"{o}rper.
\newblock \emph{Math. Ann.}, 99\penalty0 (1):\penalty0 84--117, 1928.
\newblock ISSN 0025-5831.
\newblock \doi{10.1007/BF01459087}.
\newblock URL \url{https://doi.org/10.1007/BF01459087}.

\bibitem[Prachar(1958)]{Prachar}
K.~Prachar.
\newblock Über die kleinste quadratfreie zahl einer arithmetischen reihe.
\newblock \emph{Monatshefte für Mathematik}, 62:\penalty0 173--176, 1958.
\newblock URL \url{http://eudml.org/doc/177039}.

\bibitem[Smith(2018)]{Hanson}
H.~Smith.
\newblock Two families of monogenic {$S_4$} quartic number fields.
\newblock \emph{Acta Arith.}, 186\penalty0 (3):\penalty0 257--271, 2018.
\newblock ISSN 0065-1036.
\newblock \doi{10.4064/aa180423-24-8}.
\newblock URL \url{https://doi.org/10.4064/aa180423-24-8}.

\bibitem[Westlund(1910)]{Westlund}
J.~Westlund.
\newblock On the fundamental number of the algebraic number-field {$k(\root
  p\of m)$}.
\newblock \emph{Trans. Amer. Math. Soc.}, 11\penalty0 (4):\penalty0 388--392,
  1910.
\newblock ISSN 0002-9947.
\newblock \doi{10.2307/1988640}.
\newblock URL \url{https://doi.org/10.2307/1988640}.

\end{thebibliography}
\bibliographystyle{abbrvnat}
\end{document}